\documentclass[12pt, reqno]{amsart}
\usepackage{amsmath, amsthm, amscd, amsfonts, amssymb, graphicx, color}
\usepackage[bookmarksnumbered, colorlinks, plainpages]{hyperref}

\input{mathrsfs.sty}

\newtheorem{theorem}{Theorem}[section]
\newtheorem{lemma}[theorem]{Lemma}
\newtheorem{proposition}[theorem]{Proposition}
\newtheorem{corollary}[theorem]{Corollary}
\theoremstyle{definition}

\theoremstyle{remark}
\newtheorem{remark}[theorem]{Remark}
\numberwithin{equation}{section}
\begin{document}

\title [Improvements  of Berezin number inequalities]{Improvements  of Berezin number inequalities}

\author[R. Lashkaripour, M. Bakherad, M. Hajmohamadi  ]{ Monire Hajmohamadi$^1$, Rahmatollah Lashkaripour$^2$ and Mojtaba Bakherad$^3$}

\address{$^1$$^{,2}$$^{,3}$ Department of Mathematics, Faculty of Mathematics, University of Sistan and Baluchestan, Zahedan, I.R.Iran.}

\email{$^1$monire.hajmohamadi@yahoo.com}

\email{$^{2}$lashkari@hamoon.usb.ac.ir}

\email{$^{3}$mojtaba.bakherad@yahoo.com; bakherad@member.ams.org}


\begin{abstract}
In this paper, we generalize several Berezin number inequalities involving  product of operators, which acting on a Hilbert space $\mathscr H(\Omega)$. Among other inequalities, it is shown that if $A, B$ are positive operators and $X$ is any operator, then
\begin{align*}
\textbf{ber}^{r}(H_{\alpha}(A,B))&\leq\frac{\|X\|^{r}}{2}\textbf{ber}(A^{r}+B^{r})\\&\leq\frac{\|X\|^{r}}{2}\textbf{ber}(\alpha A^{r}+(1-\alpha)B^{r})+\textbf{ber}((1-\alpha)A^{r}+\alpha B^{r}),
\end{align*}
where $H_{\alpha}(A,B)=\frac{A^\alpha XB^{1-\alpha}+A^{1-\alpha} XB^{\alpha}}{2}$, $0\leq\alpha\leq1$ and $r\geq2$.
\end{abstract} \maketitle
\section{Introduction}
Let ${\mathbb B}(\mathscr H)$ denote the $C^{*}$-algebra of all bounded linear operators on a complex Hilbert space ${\mathscr H}$ with an inner product $ \langle\, .\,,\, .\,\rangle $ and the corresponding norm $ \| \,.\, \| $. In the case when ${\rm dim}{\mathscr
H}=n$, we identify ${\mathbb B}({\mathscr H})$ with the matrix
algebra $\mathbb{M}_n$ of all $n\times n$ matrices with entries in
the complex field. An operator $A\in{\mathbb B}({\mathscr H})$ is called positive
if $\langle Ax,x\rangle\geq0$ for all $x\in{\mathscr H }$, and then we write $A\geq0$.  \\
  A functional Hilbert space $\mathscr H=\mathscr H(\Omega)$ is a Hilbert space of complex valued functions on a (nonempty) set $\Omega$, which has the property that point evaluations are continuous i.e. for each $\lambda\in \Omega$ the map $f\mapsto f(\lambda)$ is a continuous linear functional on $\mathscr H$. The Riesz representation theorem ensure that for each $\lambda\in \Omega$ there is a unique element $k_{\lambda}\in \mathscr H$ such that $f(\lambda)=\langle f,k_{\lambda}\rangle$ for all $f\in \mathscr H$. The collection $\{k_{\lambda} : \lambda\in \Omega\}$ is called the reproducing kernel of $\mathscr H$. If $\{e_{n}\}$ is an orthonormal basis for a functional Hilbert space $\mathscr H$, then the reproducing kernel of $\mathscr H$ is given by $k_{\lambda}(z)=\sum_n\overline{e_{n}(\lambda)}e_{n}(z)$; (see \cite[problem 37]{ando}). For $\lambda\in \Omega$, let $\hat{k_{\lambda}}=\frac{k_{\lambda}}{\|k_{\lambda}\|}$ be the normalized reproducing kernel of $\mathscr H$. For a bounded linear operator $A$ on $\mathscr H$, the function $\widetilde{A}$ defined on $\Omega$ by $\widetilde{A}(\lambda)=\langle A\hat{k_{\lambda}},\hat{k_{\lambda}}\rangle$ is the Berezin symbol of $A$, which firstly have been introduced by Berezin \cite{Ber1,Ber2}. Berezin set and Berezin number of the operator A are defined by
 \begin{align*}
 \textbf{Ber}(A):=\{\widetilde{A}(\lambda): \lambda\in \Omega\} \qquad \textrm{and} \qquad \textbf{ber}(A):=\sup\{|\widetilde{A}(\lambda)|: \lambda\in\Omega\},
 \end{align*}
respectively,(see \cite{kar}).
The numerical radius of $T\in\mathbb B({\mathscr H})$ is defined by
$w(A):=\sup\{| \langle Ax, x\rangle| : x\in {\mathscr H}, \Vert x \Vert=1\}.$
It is clear that
\begin{align}\label{111}
\textbf{ber}(A)\leq w(A)\leq\|A\|
\end{align}
for all $A\in \mathbb B(\mathscr H)$.
 Moreover, Berezin number of an operator $A$ satisfies the following properties:\\
$(\rm i)\textbf{ ber}(\alpha A)=|\alpha|\textbf{ber}(A)$ for all $\alpha\in \mathbb C$.\\
$(\rm ii)$ $\textbf{ber}(A+B)\leq \textbf{ber}(A)+\textbf{ber}(B)$.\\
Let $T_{i}\in \mathbb B(\mathscr H)\,\,(1\leq i\leq n)$. The generalized Euclidean Berezin number of $T_{1},...,T_{n}$ is defined in \cite{Ba} as follows
\begin{align*}
\textbf{ber}_{p}(T_{1},...,T_{n}):=\sup_{\lambda\in\Omega}\left(\sum_{i=1}^{n}\left|\langle T_{i}\hat{k}_{\lambda},\hat{k}_{\lambda}\rangle \right|^{p}\right)^{\frac{1}{p}},
\end{align*}
which has the following properties:\\
$(\rm a) \textbf{ ber}_{p}(\alpha T_{1},...,\alpha T_{n})=|\alpha|\textbf{ber}_{p}(T_{1},...,T_{n})$ for all $\alpha\in \mathbb C$;\\
$(\rm b) \textbf{ ber}_{p}(T_{1}+S_{1},...,T_{n}+S_{n})\leq\textbf{ber}_{p}(T_{1},...,T_{n})+\textbf{ber}_{p}(S_{1},...,S_{n})$,\\
where $T_{i},S_{i}\in \mathbb B(\mathscr H(\Omega)) (1\leq i\leq n)$.\\

Namely, the Berezin symbol have been investigated in detail for the Toeplitz and Hankel operators on the Hardy and Bergman spaces; it is widely applied in the various questions of analysis and uniquely determines the operator (i.e., for all $\lambda\in \Omega,  \widetilde{A}(\lambda)=\widetilde{B}(\lambda)$ implies $A=B$). For further information about Berezin symbol we refer the
reader to \cite{Ba,kar1,kar2,Nor} and references therein.\\
In \cite{Ba1}, the author showed some  Berezin number inequalities as follows:
\begin{align}\label{1}
\textbf{ber}(A^{*}XB)\leq\frac{1}{2}\textbf{ber}(B^*|X|B+A^*|X^*|A),
\end{align}
\begin{align*}
\textbf{ber}(AX\pm XA)\leq \textbf{ber}^{\frac{1}{2}}(A^*A+AA^*)\textbf{ber}^{\frac{1}{2}}(X^*X+XX^*),
\end{align*}
and
\begin{align}\label{4}
\textbf{ber}(A^{*}XB+B^*YA)\leq2\sqrt{\|X\|\|Y\|}\textbf{ber}^{\frac{1}{2}}(B^*B)\textbf{ber}^{\frac{1}{2}}(AA^*)
\end{align}
for any $A,B,X,Y\in {\mathbb B}({\mathscr H}(\Omega)).$\\
In this paper, we would like to state more extensions of Berezin number inequalities. Moreover, we obtain several Berezin number inequalities for $2\times 2$ operator matrices. For this goal we will apply some methods from \cite{haj}.
\section{main results}
To prove our Berezin number inequalities, we need several well known lemmas.\\
The following lemma is a simple consequence of the classical Jensen and Young inequalities (see \cite{har}).
 \begin{lemma}\label{3'}
 Let $a,b\geq0$, $0\leq\alpha\leq1$ and $p,q>1$ such that $\frac{1}{p}+\frac{1}{q}=1$. Then\\
 $\rm{(a)}$ $a^\alpha b^{1-\alpha}\leq\alpha a+(1-\alpha)b\leq(\alpha a^r+(1-\alpha)b^r)^{\frac{1}{r}}$ for $r\geq1$;\\
 $\rm{(b)}$ $ab\leq\frac{a^p}{p}+\frac{b^q}{q}\leq(\frac{a^{pr}}{p}+\frac{b^{qr}}{q})^{\frac{1}{r}}$
 for $r\geq1$.
 \end{lemma}
 A refinement of the Young inequality is presented in \cite{KIt} as follows:
 \begin{align}\label{refi}
 a^\alpha b^{1-\alpha}\leq\alpha a+(1-\alpha)b-r_{0}(a^{\frac{1}{2}}-b^{\frac{1}{2}})^{2},
 \end{align}
 where $r_{0}=\min\{\alpha,1-\alpha\}$.\\
The following lemma known as generalized mixed schwarz inequality \cite{kit}.
\begin{lemma}\label{mix}
Let $T\in{\mathbb B}({\mathscr H})$ and $x, y\in {\mathscr H}$ be any vectors.\\
$\rm{(a)}$ If $0\leq\alpha\leq1$, then
\begin{align*}
 \mid\langle Tx,y\rangle\mid^{2}\leq \langle\mid T\mid^{2\alpha}x,x\rangle\langle\mid T^{*}\mid^{2(1-\alpha)}y,y\rangle,
\end{align*}
where $|T|=(T^*T)^{\frac{1}{2}}$ is the absolute value of $T$.\\
$\rm{(b)}$ If $f$, $g$ are nonnegative  continuous functions on $[0, \infty)$ which are satisfying the relation $f(t)g(t)=t\,\,(t\in [0, \infty))$, then
\begin{align*}
\mid \langle Tx, y \rangle \mid \leq \parallel f(\mid T \mid)x \parallel \parallel g(\mid T^{*} \mid)y \parallel
\end{align*}
 for all $x, y\in {\mathscr H}$.
\end{lemma}
\begin{lemma}\cite{Ba}\label{9}
Let  $A\in {\mathbb B}({\mathscr H_1}(\Omega))$, $B\in {\mathbb B}({\mathscr H_2(\Omega), \mathscr H_1(\Omega)})$, $C\in {\mathbb B}({\mathscr H_1(\Omega),\mathscr H_2(\Omega)})$ and $D\in {\mathbb B}({\mathscr H_2}(\Omega))$. Then the following statements hold:\\
$(a)\,\,\textbf{ber}\left(\left[\begin{array}{cc}
 A&0\\
 0&D
 \end{array}\right]\right)$
  $\leq$ $ \max \{\textbf{ber}(A), \textbf{ber}(D)\};$
\\
\\
$(b)\,\,\textbf{ber}\left(\left[\begin{array}{cc}
 0&B\\
 C&0
 \end{array}\right]\right)$
 $\leq$ $ \frac{1}{2}(\|B\|+\|C\|).$
\end{lemma}

The next lemma follows from the spectral theorem for positive operators and the Jensen inequality (see \cite{kit}).
\begin{lemma}\label{8}
(McCarthy inequality). Let $T\in{\mathbb B}({\mathscr H})$, $ T \geq 0$ and $x\in {\mathscr H}$ be a unit vector. Then\\
$(a)\,\, \langle Tx, x\rangle^{r} \leq  \langle T^{r}x, x\rangle$ for $ r\geq 1;$\\
$(b)\,\,\langle T ^{r}x, x\rangle  \leq  \langle Tx, x\rangle^{r}$ for $ 0<r\leq 1$.
\end{lemma}
Now, by applying these lemmas, we extend some Berezin number inequalities.
\begin{theorem}
Let $A, B, X\in {\mathbb B}({\mathscr H(\Omega)})$. Then\\
$(\rm i)\,\, \textbf{ber}^{r}(A^*XB)\leq\|X\|^{r}\textbf{ber}\left(\frac{1}{p}(A^*A)^{\frac{pr}{2}}+\frac{1}{q}(B^*B)^{\frac{qr}{2}}\right)$ for $r\geq0$ and $p,q>1$ with $\frac{1}{p}+\frac{1}{q}=1$ and $pr,qr\geq2$.\\
$(\rm ii)\,\,\textbf{ber}(A^*XB)\leq\frac{1}{2}\textbf{ber}(B^*|X|^{2\alpha}B+A^*|X^*|^{2(1-\alpha)}A)$ for every $0\leq\alpha\leq1$.
\end{theorem}
\begin{proof}
If $\hat{k}_{\lambda}$ is the normalized reproducing kernel of ${\mathscr H}(\Omega)$, then
\begin{align*}
|\langle A^*XB\hat{k}_{\lambda},\hat{k}_{\lambda}\rangle|^r&=
|\langle XB\hat{k}_{\lambda},A\hat{k}_{\lambda}\rangle|^r\\&
\leq\|X\|^r\|A\hat{k}_{\lambda}\|^r\|B\hat{k}_{\lambda}\|^r \qquad\qquad (\textrm {by the Cauchy-Schwarz inequality})\\&
\leq\|X\|^r\langle A\hat{k}_{\lambda},A\hat{k}_{\lambda}\rangle^{\frac{r}{2}}\langle B\hat{k}_{\lambda},B\hat{k}_{\lambda}\rangle^{\frac{r}{2}}\\&
\leq\|X\|^{r}\left(\frac{1}{p}\langle A^*A\hat{k}_{\lambda},\hat{k}_{\lambda}\rangle^{\frac{pr}{2}}+\frac{1}{q}\langle B^*B\hat{k}_{\lambda},\hat{k}_{\lambda}\rangle^{\frac{qr}{2}}\right)\quad\quad (\textrm {by Lemma\,\,} \ref {3'})\\&
\leq\|X\|^{r}\left(\frac{1}{p}\langle (A^*A)^{\frac{pr}{2}}\hat{k}_{\lambda},\hat{k}_{\lambda}\rangle+\frac{1}{q}\langle (B^*B)^{\frac{qr}{2}}\hat{k}_{\lambda},\hat{k}_{\lambda}\rangle\right)\quad (\textrm {by Lemma\,\,}\ref{8})\\&
=\|X\|^{r}\left\langle \frac{1}{p}(A^*A)^{\frac{pr}{2}}+\frac{1}{q}(B^*B)^{\frac{qr}{2}}\hat{k}_{\lambda},\hat{k}_{\lambda}\right\rangle\\&
\leq\|X\|^r\textbf{ber}\left(\frac{1}{p}(A^*A)^{\frac{pr}{2}}+\frac{1}{q}(B^*B)^{\frac{qr}{2}}\right).
\end{align*}
Therefore
\begin{align*}
 \textbf{ber}^{r}(A^*XB)\leq\|X\|^{r}\textbf{ber}\left(\frac{1}{p}(A^*A)^{\frac{pr}{2}}+\frac{1}{q}(B^*B)^{\frac{qr}{2}}\right),
 \end{align*}
 and so we get the part $\rm (i)$.
 For the proof of the part $\rm(ii)$ we have
\begin{align*}
|\langle A^*XB\hat{k}_{\lambda},\hat{k}_{\lambda}\rangle|&=|\langle XB\hat{k}_{\lambda},A\hat{k}_{\lambda}\rangle|\\&
\leq\langle |X|^{2\alpha}B\hat{k}_{\lambda},B\hat{k}_{\lambda}\rangle^{\frac{1}{2}}\langle |X^*|^{2(1-\alpha)}A\hat{k}_{\lambda},A\hat{k}_{\lambda}\rangle^{\frac{1}{2}} \quad\quad (\textrm {by Lemma}\, \ref{mix})\\&
\leq\frac{1}{2}(\langle B^*|X|^{2\alpha}B\hat{k}_{\lambda},\hat{k}_{\lambda}\rangle+\langle A^*|X^*|^{2(1-\alpha)}A\hat{k}_{\lambda},\hat{k}_{\lambda}\rangle)\\&
 \qquad\qquad\qquad\qquad (\textrm {by the arithmetic-geometric mean inequality})\\&
=\frac{1}{2}(\langle B^*|X|^{2\alpha}B+A^*|X^*|^{2(1-\alpha)}A\hat{k}_{\lambda},\hat{k}_{\lambda}\rangle)\\&
\leq \frac{1}{2}\textbf{ber}(B^*|X|^{2\alpha}B+A^*|X^*|^{2(1-\alpha)}A).
\end{align*}
Now, the result follows by taking supremum on $\lambda\in \Omega$.
\end{proof}
\begin{theorem}\label{6}
Let $A, B, X, Y\in {\mathbb B}({\mathscr H(\Omega)})$. Then for every $0\leq\alpha\leq1$
{\footnotesize\begin{align}\label{5}
\textbf{ber}(A^*XB+B^*YA)\leq\frac{1}{2}\textbf{ber}(B^*|X|^{2\alpha}B+A^*|X^*|^{2(1-\alpha)}A+A^*|Y|^{2\alpha}A+B^*|Y^*|^{2(1-\alpha)}B).
\end{align}}
\begin{proof}
Applying Lemma \ref{mix} and the arithmetic-geometric mean inequality, for any $\hat{k}_{\lambda}\in {\mathscr H(\Omega)}$, we have
{\footnotesize\begin{align*}
|\langle &(A^*XB+B^*YA)\hat{k}_{\lambda},\hat{k}_{\lambda}\rangle|
\\&\leq|\langle A^*XB\hat{k}_{\lambda},\hat{k}_{\lambda}\rangle|+|\langle B^*YA\hat{k}_{\lambda},\hat{k}_{\lambda}\rangle|\\&
=|\langle XB\hat{k}_{\lambda},A\hat{k}_{\lambda}\rangle|+|\langle YA\hat{k}_{\lambda},B\hat{k}_{\lambda}\rangle|\\&
\leq\langle |X|^{2\alpha}B\hat{k}_{\lambda},B\hat{k}_{\lambda}\rangle^{\frac{1}{2}}\langle |X^*|^{2(1-\alpha)}A\hat{k}_{\lambda},A\hat{k}_{\lambda}\rangle^{\frac{1}{2}}+\langle |Y|^{2\alpha}A\hat{k}_{\lambda},A\hat{k}_{\lambda}\rangle^{\frac{1}{2}}\langle |Y^*|^{2(1-\alpha)}B\hat{k}_{\lambda},B\hat{k}_{\lambda}\rangle^{\frac{1}{2}}\\&
\leq\frac{1}{2}\left(\langle B^*|X|^{2\alpha}B\hat{k}_{\lambda},\hat{k}_{\lambda}\rangle+\langle A^*|X^*|^{2(1-\alpha)}A\hat{k}_{\lambda},\hat{k}_{\lambda}\rangle\right)\\&\,\,\,+\frac{1}{2}\left(\langle A^*|Y|^{2\alpha}A\widehat{k}_{\lambda},\widehat{k}_{\lambda}\rangle+\langle B^*|Y^*|^{2(1-\alpha)}B\hat{k}_{\lambda},\hat{k}_{\lambda}\rangle\right)\\&
=\frac{1}{2}\left(\langle (B^*|X|^{2\alpha}B+A^*|X^*|^{2(1-\alpha)}A+A^*|Y|^{2\alpha}A+B^*|Y^*|^{2(1-\alpha)}B)\hat{k}_{\lambda},\hat{k}_{\lambda}\rangle\right)\\&
\leq\frac{1}{2}\textbf{ber}(B^*|X|^{2\alpha}B+A^*|X^*|^{2(1-\alpha)}A+A^*|Y|^{2\alpha}A+B^*|Y^*|^{2(1-\alpha)}B).
\end{align*}}
Now, by taking supremum over $\lambda\in \Omega$, we get the desired inequality.

\end{proof}
\end{theorem}
Inequality \eqref{5} yields several Berezin number inequalities as special cases. A sample of elementary inequalities are demonstrated in the following remarks.
\begin{remark}
By letting $\alpha=\frac{1}{2}$ in inequality \eqref{5}, we get the following inequalities:
\begin{align*}
\textbf{ber}(A^*XB+B^*YA)&\leq\frac{1}{2}\textbf{ber}(B^*|X|B+A^*|X^*|A+A^*|Y|A+B^*|Y^*|B)\\&
\leq\frac{1}{2}\textbf{ber}(B^*|X|B+A^*|X^*|A)+\frac{1}{2}\textbf{ber}(A^*|Y|A+B^*|Y^*|B).
\end{align*}
\end{remark}
\begin{remark}
Putting $\alpha=\frac{1}{2}$, $A=I$ and $X=Y=A$, in inequality \eqref{5}  we can obtain the following inequality.
\begin{align*}
\textbf{ber}(AB+B^*A)\leq\frac{1}{2}\textbf{ber}(|A|+|A^*|)+\frac{1}{2}\textbf{ber}(B^*(|A|+|A^*|)B).
\end{align*}
\end{remark}
In the next result we find an upper bound for power of the  Berezin number of $A^{\alpha}XB^{1-\alpha}$, in which  $0\leq\alpha\leq1$.
\begin{theorem}\label{11}
Suppose that $A,B,X\in{\mathbb B}({\mathscr H(\Omega)})$ such that $A, B$ are positive. Then
\begin{align}\label{10}
\textbf{ber}^{r}(A^{\alpha}XB^{1-\alpha})\leq\|X\|^{r}\left(\textbf{ber}(\alpha A^{r}+(1-\alpha)B^{r})-\inf_{\|\hat{k}_{\lambda}\|=1}\eta(\hat{k}_{\lambda})\right),
\end{align}
in which $\eta(\hat{k}_{\lambda})=r_{0}(\langle A^{r}\hat{k}_{\lambda},\hat{k}_{\lambda}\rangle^{\frac{1}{2}}-\langle B^{r}\hat{k}_{\lambda},\hat{k}_{\lambda}\rangle^{\frac{1}{2}})^{2}$, $r_{0}=\min\{\alpha,1-\alpha\}$, $r\geq 2$ and $0\leq\alpha\leq1$.
\end{theorem}
\begin{proof}
Let $\hat{k}_{\lambda}$ be the normalized reproducing kernel of ${\mathscr H}(\Omega)$. Then
{\footnotesize\begin{align*}
| \langle A^{\alpha}XB^{1-\alpha}\hat{k}_{\lambda},\hat{k}_{\lambda}\rangle |^{r}&=
|\langle XB^{1-\alpha}\hat{k}_{\lambda},A^{\alpha}\hat{k}_{\lambda}\rangle|^r\\&
\leq\|x\|^r\|B^{1-\alpha}\hat{k}_{\lambda}\|^{r}\|A^{\alpha}\hat{k}_{\lambda}\|^{r} \quad(\textrm {by the Cauchy-Schwarz inequality})\\&
=\|X\|^r\langle B^{2(1-\alpha)}\hat{k}_{\lambda},\hat{k}_{\lambda}\rangle^{\frac{r}{2}}\langle A^{2\alpha}\hat{k}_{\lambda},\hat{k}_{\lambda}\rangle^{\frac{r}{2}}\\&
\leq\|X\|^r\langle A^{r}\hat{k}_{\lambda},\hat{k}_{\lambda}\rangle^{\alpha}\langle B^{r}\hat{k}_{\lambda},\hat{k}_{\lambda}\rangle^{1-\alpha}\qquad(\textrm {by Lemma\,\,}\ref{8})\\&
\leq\|X\|^r\left(\langle(\alpha A^{r}+(1-\alpha)B^{r})\hat{k}_{\lambda},\hat{k}_{\lambda}\rangle-r_{0}(\langle A^{r}\hat{k}_{\lambda},\hat{k}_{\lambda}\rangle^{\frac{1}{2}}-\langle B^{r}\hat{k}_{\lambda},\hat{k}_{\lambda}\rangle^{\frac{1}{2}})^{2}\right)
\\&
\qquad\qquad\qquad\qquad\qquad(\textrm {by \,\,}\eqref{refi})\\&
=\|X\|^r(\langle(\alpha A^{r}+(1-\alpha)B^{r})\hat{k}_{\lambda},\hat{k}_{\lambda}\rangle)-\|X\|^{r}r_{0}(\langle A^{r}\hat{k}_{\lambda},\hat{k}_{\lambda}\rangle^{\frac{1}{2}}-\langle B^{r}\hat{k}_{\lambda},\hat{k}_{\lambda}\rangle^{\frac{1}{2}})^{2}\\&
\leq\|X\|^r\left(\textbf{ber}(\alpha A^{r}+(1-\alpha)B^{r})-r_{0}(\langle A^{r}\hat{k}_{\lambda},\hat{k}_{\lambda}\rangle^{\frac{1}{2}}-\langle B^{r}\hat{k}_{\lambda},\hat{k}_{\lambda}\rangle^{\frac{1}{2}})^{2}\right).
\end{align*}}
Taking the supremum over $\lambda\in \Omega$, we deduce the desired result.
\end{proof}
\begin{remark}
Putting $A=B=I$ in inequality \eqref{10}, we get a generalization of the inequality \eqref{111}.
\end{remark}
The Heinz mean is defined as $H_{\alpha}(a,b)=\frac{a^{1-\alpha}b^\nu+a^\alpha b^{1-\alpha}}{2}$ for $a,b>0$ and $0\leq\alpha\leq1$. The function $H_{\alpha}$ is symmetric about the point $\nu={1\over2}$ and  $\sqrt{ab}\leq H_\alpha(a,b)\leq {a+b\over2}$
 for all $\alpha\in[0,1]$. For further information about the Heinz mean we refer the
reader to \cite{bakh, lashkari3} and references therein.
In the next theorem we can obtain an upper bound for the Berezin number involving power Heinz mean.
\begin{theorem}
Suppose that $A,B,X\in{\mathbb B}({\mathscr H(\Omega)})$ such that $A,B$ are positive. Then
\begin{align*}
\textbf{ber}^{r}\Big(\frac{A^\alpha XB^{1-\alpha}+A^{1-\alpha} XB^{\alpha}}{2}&\Big)\leq\frac{\|X\|^{r}}{2}\textbf{ber}(A^{r}+B^{r})\\&\leq\frac{\|X\|^{r}}{2}\textbf{ber}(\alpha A^{r}+(1-\alpha)B^{r})+\textbf{ber}((1-\alpha)A^{r}+\alpha B^{r}),
\end{align*}
in which $r\geq 2$ and $0\leq\alpha\leq1$.
\end{theorem}
\begin{proof}
Using Theorem \ref{11} for $\hat{k}_{\lambda}$, which is the normalized reproducing kernel of ${\mathscr H}(\Omega)$ we have
\begin{align*}
&\left|\left\langle\frac{A^\alpha XB^{1-\alpha}+A^{1-\alpha} XB^{\alpha}}{2}\hat{k}_{\lambda},\hat{k}_{\lambda}\right\rangle\right|^{r}\\&\leq\left(\frac{|\langle A^\alpha XB^{1-\alpha}\hat{k}_{\lambda},\hat{k}_{\lambda}\rangle|+|\langle A^{1-\alpha} XB^{\alpha}\hat{k}_{\lambda},\hat{k}_{\lambda}\rangle|}{2}\right)^{r}\\&
\leq\frac{1}{2}(|\langle A^\alpha XB^{1-\alpha}\hat{k}_{\lambda},\hat{k}_{\lambda}\rangle|^{r}+|\langle A^{1-\alpha} XB^{\alpha}\hat{k}_{\lambda},\hat{k}_{\lambda}\rangle|^{r})\qquad(\textrm {by the convexity of\,\,} {f(t)=t^{r}})\\&
\leq\frac{\|X\|^{r}}{2}(\langle\alpha A^r+(1-\alpha)B^{r}\hat{k}_{\lambda},\hat{k}_{\lambda}\rangle+\langle(1-\alpha) A^r+\alpha B^{r}\hat{k}_{\lambda},\hat{k}_{\lambda}\rangle)\\&
=\frac{\|X\|^{r}}{2}\langle(A^r+B^r)\hat{k}_{\lambda},\hat{k}_{\lambda}\rangle\\&
\leq\frac{\|X\|^{r}}{2}\textbf{ber}(A^r+B^r).
\end{align*}
Taking supremum over $\lambda\in\Omega$, we get the first inequality.
For the second inequality, we have
\begin{align*}
\frac{\|X\|^{r}}{2}\textbf{ber}(A^r+B^r)&=\frac{\|X\|^{r}}{2}\textbf{ber}(\alpha A^r+(1-\alpha)B^r+(1-\alpha)A^r+\alpha B^r)\\&
\leq\frac{\|X\|^{r}}{2}\textbf{ber}(\alpha A^r+(1-\alpha)B^r)+\textbf{ber}((1-\alpha)A^r+\alpha B^r).
\end{align*}
\end{proof}

\section{Berezin number inequalities of off-diagonal matrices}
In this section, we obtain some inequalities involving powers of the Berezin number for the off-diagonal parts of $2\times2$ operator matrices.

\begin{theorem}
Let
$T=\left[\begin{array}{cc}
 0&B\\
 C&0
 \end{array}\right]\in {\mathbb B}({\mathscr H_2,\mathscr H_1})$ and  $f$, $g$ be nonnegative  continuous  functions on $[0, \infty)$ satisfying the relation $f(t)g(t)=t\,\,(t\in [0, \infty))$. Then
 {\footnotesize\begin{align}\label{7}
\textbf{ber}^{r}(T)\leq \max \left\{\textbf{ ber}\left( \frac{1}{p} f^{pr} (\mid C \mid) + \frac{1}{q} g^{qr} (\mid B^{*} \mid) \right), \textbf{ber}\left(\frac{1}{p} f^{pr}(\mid B \mid)+\frac{1}{q} g^{qr}(\mid C^{*}\mid) \right)\right\},
   \end{align}}
 in which  $r\geq 1$, $p \geq q >1$ such that $\frac{1}{p}+\frac{1}{q}=1$ and $pr \geq 2$.\\
\end{theorem}
\begin{proof}
For any $(\lambda_{1},\lambda_{2})\in \Omega_{1}\times\Omega_{2}$, let $\hat{k}_{(\lambda_{1},\lambda_{2})}=\left[\begin{array}{cc}
 k_{\lambda_{1}}\\
 k_{\lambda_{2}}
 \end{array}\right]$ be the normalized reproducing kernel in $\mathscr H(\Omega_{1})\oplus\mathscr H(\Omega_{2})$. Then
 {\scriptsize\begin{align*}
 &\mid \langle T\hat{k}(\lambda_{1},\lambda_{2}),\hat{k}(\lambda_{1},\lambda_{2}) \rangle \mid^{r}
 \\&\leq \parallel f\left(\mid T \mid\right)\hat{k}(\lambda_{1},\lambda_{2})\parallel^{r} \parallel g\left(\mid T^{*}\mid \right)\hat{k}(\lambda_{1},\lambda_{2})\parallel ^{r}
    \qquad ( \normalsize\textrm{by Lemma \ref{mix}(b)})
  \\&= \langle f^{2}(\mid T \mid)\hat{k}(\lambda_{1},\lambda_{2}),\hat{k}(\lambda_{1},\lambda_{2})\rangle^{\frac{r}{2}} \langle g^{2}(\mid T^{*} \mid)\hat{k}(\lambda_{1},\lambda_{2}),\hat{k}(\lambda_{1},\lambda_{2})\rangle^{\frac{r}{2}}
 \\&
 \leq \frac{1}{p}\left\langle f^{2}\left(
 \left[\begin{array}{cc}
 \mid C \mid & 0\\
 0 & \mid B \mid
 \end{array}\right]
 \right)
 \hat{k}(\lambda_{1},\lambda_{2}),\hat{k}(\lambda_{1},\lambda_{2})\right\rangle ^{\frac{pr}{2}}+\frac{1}{q} \left\langle g^{2}\left(
 \left[\begin{array}{cc}
 \mid B^{*} \mid & 0\\
 0 & \mid C^{*} \mid
 \end{array}\right]
 \right)
 \hat{k}(\lambda_{1},\lambda_{2}),\hat{k}(\lambda_{1},\lambda_{2})\right\rangle ^{\frac{qr}{2}}  \\&     \qquad \qquad \qquad     \qquad (\normalsize\textrm{ by Lemma \ref{3'}(b)})
 \\&
 \leq \frac{1}{p}\left\langle \left[\begin{array}{cc}
 f^{pr} \mid C \mid & 0\\
 0 & f^{pr} \mid B \mid
 \end{array}\right]
\hat{k}(\lambda_{1},\lambda_{2}),\hat{k}(\lambda_{1},\lambda_{2})\right\rangle +\frac{1}{q} \left\langle \left[\begin{array}{cc}
 g^{qr} \mid B^{*} \mid & 0\\
 0 & g^{qr} \mid C^{*} \mid
 \end{array}\right]
 \hat{k}(\lambda_{1},\lambda_{2}),\hat{k}(\lambda_{1},\lambda_{2})\right\rangle\\&\qquad   \qquad \qquad \qquad             ( \normalsize\textrm{by Lemma \ref{8}(a)})
 \\&
 =\left\langle \left[\begin{array}{cc}
 \frac{1}{p} f^{pr} (\mid C \mid) +\frac{1}{q} g^{qr} (\mid B^{*}\mid) & 0\\
 0 & \frac{1}{p}f^{pr} (\mid B \mid)+\frac{1}{q} g^{qr}(\mid C^{*} \mid)
 \end{array}\right]
 \hat{k}(\lambda_{1},\lambda_{2}),\hat{k}(\lambda_{1},\lambda_{2})\right\rangle.
  \end{align*}}
 Hence
 {\scriptsize\begin{align*}
 \mid \langle T\hat{k}(\lambda_{1},\lambda_{2})&,\hat{k}(\lambda_{1},\lambda_{2}) \rangle\mid^{r} \\&\leq \textbf{ber}\left( \left\langle \left[\begin{array}{cc}
 \frac{1}{p} f^{pr} (\mid C \mid) +\frac{1}{q} g^{qr} (\mid B^{*}\mid) & 0\\
 0 & \frac{1}{p}f^{pr} (\mid B \mid)+\frac{1}{q} g^{qr}(\mid C^{*} \mid)
 \end{array}\right]
 \hat{k}(\lambda_{1},\lambda_{2}),\hat{k}(\lambda_{1},\lambda_{2})\right\rangle\right) .
 \end{align*}}
  Now, applying the definition of Berezin number and Lemma \ref{9}(a), we have
 {\footnotesize\begin{align*}
\textbf{ber}^{r}(T)\leq \max \left\{ \textbf{ber}\left( \frac{1}{p} f^{pr} (\mid C \mid) + \frac{1}{q} g^{qr} (\mid B^{*} \mid) \right), \textbf{ber}\left(\frac{1}{p} f^{pr}(\mid B \mid)+\frac{1}{q} g^{qr}(\mid C^{*}\mid) \right)\right\}.
 \end{align*}}

\end{proof}
Inequality \eqref{7} induces the following inequality.
\begin{corollary}
Let
$T=\left[\begin{array}{cc}
 0&B\\
 C&0
 \end{array}\right]\in {\mathbb B}({\mathscr H_2,\mathscr H_1})$. Then
{\scriptsize\begin{align*}
\textbf{ber}^{r}(T)\leq \frac{1}{2}\max\{\textbf{ber}( \mid C \mid^{2r\alpha}+\mid B^{*} \mid^{2r(1-\alpha)}),\textbf{ber}( \mid B \mid^{2r\alpha}+\mid C^{*} \mid^{2r(1-\alpha)}) \}
 \end{align*}}
  for any $r\geq 1$ and $0\leq\alpha\leq 1$.
\end{corollary}
\begin{proof}
Letting $f(t)=t^{\alpha}$, $g(t)=t^{1-\alpha}$ and $p=q=2$ in inequality \eqref{7}, we get the desired inequality.
\end{proof}

\begin{theorem}
 Let
 $T_{i}=\left[\begin{array}{cc}
 0&B_{i}\\
 C_{i}&0
 \end{array}\right]
 \in {\mathbb B}({\mathscr H_2(\Omega)\oplus\mathscr H_1(\Omega)})$ for any $i=1, 2,\cdots,n$. Then
 {\tiny\begin{align}
 \textbf{ber}_{p}^{p}(T_{1}, T_{2},\cdots,T_{n})\leq \max\left\{\textbf{ber}\left( \sum_{i=1}^{n} \alpha \mid C_{i}\mid^{p}+(1-\alpha)\mid B_{i}^{*}\mid^{p}\right), \textbf{ber}\left( \sum_{i=1}^{n} \alpha \mid B_{i}\mid^{p}+(1-\alpha)\mid C_{i}^{*}\mid^{p}\right)\right\}
 \end{align}}
 for $0\leq \alpha\leq 1 $ and $p\geq 2$.
 \end{theorem}
  \begin{proof}
 For any $(\lambda_{1},\lambda_{2})\in \Omega_{1}\times\Omega_{2}$, let $\hat{k}_{(\lambda_{1},\lambda_{2})}=\left[\begin{array}{cc}
 k_{\lambda_{1}}\\
 k_{\lambda_{2}}
 \end{array}\right]$ be the normalized reproducing kernel in $\mathscr H(\Omega_{1})\oplus\mathscr H(\Omega_{2})$. Then
 {\footnotesize\begin{align*}
 &\sum_{i=1}^{n}\mid\langle T_{i}\hat{k}_{(\lambda_{1},\lambda_{2})},\hat{k}_{(\lambda_{1},\lambda_{2})}\rangle\mid^{p}
 \\&=\sum_{i=1}^{n}(\mid\langle T_{i}\hat{k}_{(\lambda_{1},\lambda_{2})},\hat{k}_{(\lambda_{1},\lambda_{2})}\rangle\mid^{2})^{\frac{p}{2}}\\&
 \leq \sum_{i=1}^{n}(\langle\mid T_{i}\mid^{2\alpha}\hat{k}_{(\lambda_{1},\lambda_{2})},\hat{k}_{(\lambda_{1},\lambda_{2})}\rangle \langle \mid T_{i}^{*}\mid^{2(1-\alpha)}\hat{k}_{(\lambda_{1},\lambda_{2})},\hat{k}_{(\lambda_{1},\lambda_{2})}\rangle)^{\frac{p}{2}}   \qquad \normalsize{(\textrm {by Lemma  \ref{mix}(a)})}\\&
 \leq \sum_{i=1}^{n}\langle \mid T_{i}\mid^{p\alpha}\hat{k}_{(\lambda_{1},\lambda_{2})},\hat{k}_{(\lambda_{1},\lambda_{2})}\rangle\langle \mid T_{i}^{*}\mid^{p(1-\alpha)}\hat{k}_{(\lambda_{1},\lambda_{2})},\hat{k}_{(\lambda_{1},\lambda_{2})}\rangle
  \qquad\qquad (\normalsize\textrm {by Lemma \ref{8}(a)})     \\&
 \leq \sum_{i=1}^{n}\langle \mid T_{i}\mid^{p}\hat{k}_{(\lambda_{1},\lambda_{2})},\hat{k}_{(\lambda_{1},\lambda_{2})}\rangle^{\alpha}\langle\mid T_{i}^{*}\mid^{p}\hat{k}_{(\lambda_{1},\lambda_{2})},\hat{k}_{(\lambda_{1},\lambda_{2})}\rangle^{1-\alpha} \qquad  \qquad (\normalsize\textrm {by Lemma   \ref{8}(b)})\\&
 \leq\sum_{i=1}^{n}(\alpha \langle\mid T_{i}\mid^{p}\hat{k}_{(\lambda_{1},\lambda_{2})},\hat{k}_{(\lambda_{1},\lambda_{2})}\rangle+(1-\alpha)\langle \mid T_{i}^{*}\mid^{p}\hat{k}_{(\lambda_{1},\lambda_{2})},\hat{k}_{(\lambda_{1},\lambda_{2})}\rangle)
 \qquad (\normalsize\textrm {by Lemma \ref{3'}(a)})  \\&
 =\sum_{i=1}^{n}\left(
 \alpha\left\langle
 \left[\begin{array}{cc}
 \mid C_{i}\mid^{p}&0\\
 0&\mid B_{i}\mid^{p}
 \end{array}\right]
 \hat{k}_{(\lambda_{1},\lambda_{2})},\hat{k}_{(\lambda_{1},\lambda_{2})}\right\rangle + (1-\alpha)\left\langle
 \left[\begin{array}{cc}
 \mid B_{i}^{*}\mid^{p}&0\\
 0&\mid C_{i}^{*}\mid^{p}
 \end{array}\right]
 \hat{k}_{(\lambda_{1},\lambda_{2})},\hat{k}_{(\lambda_{1},\lambda_{2})}\right\rangle
\right)
\\&
=\sum_{i=1}^{n} \left\langle
\left[\begin{array}{cc}
 \alpha\mid C_{i}\mid^{p}+(1-\alpha)\mid B_{i}^{*}\mid^{p}&0\\
 0&\alpha \mid B_{i}\mid^{p}+(1-\alpha)\mid C_{i}^{*}\mid^{p}
 \end{array}\right]
\hat{k}_{(\lambda_{1},\lambda_{2})},\hat{k}_{(\lambda_{1},\lambda_{2})}\right\rangle\\&
 =\left\langle
 \left[\begin{array}{cc}
 \sum_{i=1}^{n}\alpha\mid C_{i}\mid^{p}+(1-\alpha)\mid B_{i}^{*}\mid^{p}&0\\
 0& \sum_{i=1}^{n}\alpha \mid B_{i}\mid^{p}+(1-\alpha)\mid C_{i}^{*}\mid^{p}
 \end{array}\right]
\hat{k}_{(\lambda_{1},\lambda_{2})},\hat{k}_{(\lambda_{1},\lambda_{2})}\right\rangle.
 \end{align*}}
 Therefore
 {\footnotesize\begin{align*}
 \sum_{i=1}^{n}\mid&\langle T_{i}\hat{k}_{(\lambda_{1},\lambda_{2})},\hat{k}_{(\lambda_{1},\lambda_{2})}\rangle\mid^{p}\\&
 \leq
 \textbf{ber}\left\langle
 \left[\begin{array}{cc}
 \sum_{i=1}^{n}\alpha\mid C_{i}\mid^{p}+(1-\alpha)\mid B_{i}^{*}\mid^{p}&0\\
 0& \sum_{i=1}^{n}\alpha \mid B_{i}\mid^{p}+(1-\alpha)\mid C_{i}^{*}\mid^{p}
 \end{array}\right]
\hat{k}_{(\lambda_{1},\lambda_{2})},\hat{k}_{(\lambda_{1},\lambda_{2})}\right\rangle.
 \end{align*}}
 By the definition of the Berezin number and Lemma \ref{9}(a), we get
 {\footnotesize\begin{align*}
 \textbf{ber}_{p}^{p}(T_{1}, T_{2},&\cdots,T_{n})\\&\leq \max\left\{\textbf{ber}\left( \sum_{i=1}^{n} \alpha \mid C_{i}\mid^{p}+(1-\alpha)\mid B_{i}^{*}\mid^{p}\right),\textbf{ber}\left( \sum_{i=1}^{n} \alpha \mid B_{i}\mid^{p}+(1-\alpha)\mid C_{i}^{*}\mid^{p}\right)\right\}.
 \end{align*}}
 \end{proof}
Now, we would like to estimate the Berezin number for matrix $\left[\begin{array}{cc}
  A&B\\
 C&D
 \end{array}\right].$
  \begin{proposition}
 Let $ T=\left[\begin{array}{cc}
  A&0\\
 0&D
 \end{array}\right] \in {\mathbb B}({\mathscr H_{1}(\Omega)}\oplus {\mathscr H_{2}(\Omega)})$. Then
 \begin{align}\label{14}
 \textbf{ber}^{r}(T)\leq \frac{1}{2}\max\{ \textbf{ber}(|A|^{r}+|A^{*}|^{r}),\textbf{ber}(|D|^{r}+|D^{*}|^{r})\}
 \end{align}
 for $r\geq 1$.
 \end{proposition}
 \begin{proof}
 Let $\hat{k}_{(\lambda_{1},\lambda_{2})}=\left[\begin{array}{cc}
 k_{\lambda_{1}}\\
 k_{\lambda_{2}}
 \end{array}\right]$ be the normalized reproducing kernel of $\mathscr H(\Omega_{1})\oplus\mathscr H(\Omega_{2})$. Then
 {\footnotesize\begin{align*}
 |\langle T&\hat{k}_{(\lambda_{1},\lambda_{2})},\hat{k}_{(\lambda_{1},\lambda_{2})}\rangle|\\&
 \leq \langle |T|\hat{k}_{(\lambda_{1},\lambda_{2})},\hat{k}_{(\lambda_{1},\lambda_{2})}\rangle^{\frac{1}{2}}\langle |T^{*}|\hat{k}_{(\lambda_{1},\lambda_{2})},\hat{k}_{(\lambda_{1},\lambda_{2})}\rangle^{\frac{1}{2}}\\&
 \leq \frac{1}{2}\left\langle \left[\begin{array}{cc}
 \mid A\mid&0\\
 0&\mid D \mid
 \end{array}\right]\hat{k}_{(\lambda_{1},\lambda_{2})},\hat{k}_{(\lambda_{1},\lambda_{2})} \right\rangle+\frac{1}{2}\left\langle \left[\begin{array}{cc}
 \mid A^{*}\mid&0\\
 0&\mid D^{*} \mid
 \end{array}\right]\hat{k}_{(\lambda_{1},\lambda_{2})},\hat{k}_{(\lambda_{1},\lambda_{2})} \right\rangle\\&
 \leq \left(\frac{1}{2}\left\langle \left[\begin{array}{cc}
 \mid A\mid&0\\
 0&\mid D \mid
 \end{array}\right]\hat{k}_{(\lambda_{1},\lambda_{2})},\hat{k}_{(\lambda_{1},\lambda_{2})} \right\rangle^{r}+\frac{1}{2}\left\langle \left[\begin{array}{cc}
 \mid A^{*}\mid&0\\
 0&\mid D^{*} \mid
 \end{array}\right]\hat{k}_{(\lambda_{1},\lambda_{2})},\hat{k}_{(\lambda_{1},\lambda_{2})} \right\rangle^{r}\right)^{\frac{1}{r}}\\&
 \leq \left(\frac{1}{2}\left\langle \left[\begin{array}{cc}
 \mid A\mid^{r}&0\\
  0&\mid D \mid^{r}
 \end{array}\right]\hat{k}_{(\lambda_{1},\lambda_{2})},\hat{k}_{(\lambda_{1},\lambda_{2})} \right\rangle+\frac{1}{2}\left\langle \left[\begin{array}{cc}
 \mid A^{*}\mid^{r}&0\\
 0&\mid D^{*} \mid^{r}
 \end{array}\right]\hat{k}_{(\lambda_{1},\lambda_{2})},\hat{k}_{(\lambda_{1},\lambda_{2})} \right\rangle\right)^{\frac{1}{r}}\\&
 =\left(
 \left\langle \left[\begin{array}{cc}
 \frac{1}{2}(\mid A\mid^{r}+|A^{*}|^{r})&0\\
 0&\frac{1}{2}(\mid D \mid^{r}+|D^{*}|^{r})
 \end{array}\right]\hat{k}_{(\lambda_{1},\lambda_{2})},\hat{k}_{(\lambda_{1},\lambda_{2})} \right\rangle\right)^{\frac{1}{r}}.
 \end{align*}}
 Thus
 \begin{align*}
 |\langle T\hat{k}_{(\lambda_{1},\lambda_{2})},\hat{k}_{(\lambda_{1},\lambda_{2})}\rangle |^{r}\leq \left\langle \left[\begin{array}{cc}
 \frac{1}{2}(\mid A\mid^{r}+|A^{*}|^{r})&0\\
 0&\frac{1}{2}(\mid D \mid^{r}+|D^{*}|^{r})
 \end{array}\right]\hat{k}_{(\lambda_{1},\lambda_{2})},\hat{k}_{(\lambda_{1},\lambda_{2})} \right\rangle.
 \end{align*}
 Therefore
 \begin{align*}
 \textbf{ber}^{r}(T)\leq \frac{1}{2}\max\{\textbf{ber}(|A|^{r}+|A^{*}|^{r}),\textbf{ber}(|D|^{r}+|D^{*}|^{r})\}.
 \end{align*}
 \end{proof}
 The following corollary deduces from inequalities \eqref{7} and \eqref{14} directly.
  \begin{corollary}
 Let  $T=\left[\begin{array}{cc}
  A&B\\
 C&D
 \end{array}\right]$ with $A, B, C, D\in {\mathbb B}({\mathscr H})$. Then
 {\footnotesize\begin{align*}
 \textbf{ber}(T)\leq \frac{1}{2}\max\{\textbf{ber}(|C| +| B^{*}|),\textbf{ber}(| B| + |C^{*}|)\}+\frac{1}{2}\max\{\textbf{ber}(|A| +|A^{*}|),\textbf{ber}(|D|+ |D^{*}|) \}.
 \end{align*}}
 In particular,
 \begin{align*}
  \textbf{ber}\left(\left[\begin{array}{cc}
  A&B\\
 B&A
 \end{array}\right]\right)
 \leq \frac{1}{2}(\textbf{ber}(|A|+ |A^{*}|)+\textbf{ber}(|B|+|B^{*}|)).\\
 \end{align*}
 \end{corollary}

\bigskip
\bibliographystyle{amsplain}

\begin{thebibliography}{99}

\bibitem{Ba}  M. Bakherad, \textit{Some Berezin number inequalities for operator matrices},
 Czechoslovak Math. J. (to appear).

\bibitem{bakh} M. Bakherad and M.S. Moslehian, \textit{Reverses and variations of Heinz inequality},
Linear Multilinear Algebra 63 (2015), no. 10, 1972-1980.

\bibitem{Ba1}  M. Bakherad and M.T. Karaev, \textit{ Berezin number inequalities for Hilbert space operators},
	arXiv:1805.01018.

\bibitem{Ber1}  F.A. Berezin, \textit{Covariant and contravariant symbols for operators,}
 Math. USSR-Izv.  \textbf{6} (1972), 1117–-1151.

\bibitem{Ber2} F.A. Berezin, \textit{Quantizations},
 Math. USSR-Izv. \textbf{8} (1974),  1109–-1163.

\bibitem{lashkari3} M. Hajmohamadi, R. Lashkaripour and M. Bakherad, \textit{Some  generalizations  of  numerical
radius on off--diagonal part of $2 \times 2$ operator matrices.}
  J. Math. Inequal. (to appear).

\bibitem{haj} M. Hajmohamadi, R. Lashkaripour and M. Bakherad,  \textit{Some extensions of the Young and  Heinz inequalities for Matrices},
Bull. Iranian Math. Soc. (to appear).

\bibitem{ando}P.R. Halmos, \textit{A Hilbert Space Problem Book},
2nd ed., springer, New York, 1982.

\bibitem{har}G.H. Hardy, J.E. Littlewood, G. Polya, \textit{Inequalities},
2nd ed., Cambridge Univ. Press, Cambridge, 1988.

\bibitem{kar}  M.T. Karaev, \textit{ Berezin symbol and invertibility of operators on the functional Hilbert spaces},
J. Funct. Anal. \textbf{238} (2006),  181–-192.
\bibitem{kar1} M.T. Karaev, \textit{ Functional analysis proofs of Abels theorems},
 Proc. Amer. Math. Soc. \textbf{132} (2004), 2327–-2329.

\bibitem{kar2}  M.T. Karaev and S. Saltan, \textit{Some results on Berezin symbols},
Complex Var. Theory Appl. \textbf{50} (3) (2005), 185--193.

\bibitem{kit}  F. Kittaneh, \textit{Notes on some inequalities for Hilbert space operators},
Publ. Res. Inst. Math. Sci. \textbf{24} (1988)  283–-293.

\bibitem{KIt}  F. Kittaneh and Y. Manasrah, \textit{Improved Young and Heinz inequalities for matrices},
 J. Math. Anal. Appl. \textbf{361} (2010)  262–-269.

\bibitem{Nor}  E. Nordgren and P. Rosenthal, \textit{Boundary values of Berezin symbols},
Oper. Theory Adv. Appl. \textbf{73} (1994), 362–-368.


\end{thebibliography}

\end{document}